\documentclass[a4paper,reqno]{amsart}

\usepackage{amsmath}
\usepackage{amsthm}
\usepackage{amssymb}
\usepackage{eucal}
\usepackage{tikz-cd}
\usepackage[colorlinks]{hyperref}

\numberwithin{equation}{section}
\usetikzlibrary{cd,arrows}
\tikzset{vertex/.style={shape=circle,draw,minimum size=3pt,inner sep=0pt}}
\tikzset{black-vertex/.style={shape=circle,draw,inner sep=3pt}}
\tikzset{red-vertex/.style={shape=circle,draw=red,dashed,inner sep=3pt}}
\tikzset{edge/.style={->,>=latex'}}

\theoremstyle{plain}

\newtheorem{thm}{Theorem}[section]
\newtheorem{prop}[thm]{Proposition}
\newtheorem{cor}[thm]{Corollary}
\newtheorem{lem}[thm]{Lemma}

\theoremstyle{definition}
\newtheorem{defn}{Definition}[section]

\newtheorem{exam}[defn]{Example}
\theoremstyle{remark}
\newtheorem{rmk}{Remark}[section]

\DeclareMathOperator{\Hom}{Hom}
\DeclareMathOperator{\Ext}{Ext}

\DeclareMathOperator{\End}{End}

\DeclareMathOperator{\oHom}{\overline{Hom}}

\DeclareMathOperator{\Ker}{Ker}
\DeclareMathOperator{\Coker}{Coker}

\renewcommand{\top}{\operatorname{top}}
\DeclareMathOperator{\rad}{rad}
\DeclareMathOperator{\soc}{soc}
\renewcommand{\mod}{\operatorname{mod}}
\DeclareMathOperator{\grade}{grade}
\DeclareMathOperator{\depth}{depth}
\DeclareMathOperator{\pdim}{pdim}
\DeclareMathOperator{\idim}{idim}
\DeclareMathOperator{\fdim}{fdim}
\DeclareMathOperator{\gldim}{gldim}
\DeclareMathOperator{\op}{op}
\DeclareMathOperator{\Tr}{Tr}
\DeclareMathOperator{\add}{add}
\DeclareMathOperator{\occ}{occ}

\newcommand{\lto}{\longrightarrow}

\newcommand{\fkp}{\mathfrak{p}}

\newcommand{\fkm}{\mathfrak{m}}

\newcommand{\scB}{\mathcal{B}}

\newcommand{\scE}{\mathcal{E}}

\newcommand{\scJ}{\mathcal{J}}

\newcommand{\scO}{\mathcal{O}}
\newcommand{\scP}{\mathcal{P}}

\newcommand{\scS}{\mathcal{S}}

\begin{document}

\title[A Serre-type criterion for n-Gorenstein rings]
{A Serre-type criterion for $n$-Gorenstein rings 
and its application to Nakayama algebras}

\author{Dawei Shen}
\address{School of Mathematics and Statistics \\
         Henan University \\
         Kaifeng, Henan 475004 \\
         P. R. China}
\email{sdw12345@mail.ustc.edu.cn}
\date{\today}
\begin{abstract}
Let $R$ be a two-sided Noetherian ring.
We introduce a Serre-type condition $(G_n)$, 
formulated in terms of the occurrence degree and 
the flat dimension of indecomposable injective $R$-modules,
and prove that $R$ is $n$-Gorenstein if and only if it satisfies $(G_n)$.
We then apply this criterion to Nakayama algebras via syzygy filtration.
It turns out that syzygy filtration preserves the Auslander-Gorenstein property 
and reflects it within the class of
$2$-Gorenstein Nakayama algebras.
As an application, we obtain a module-theoretic strengthening 
of the dimension bound conjecture of Klász, Kleinau and Marczinzik.
\end{abstract}

\maketitle

\section{Introduction}
Auslander-Gorenstein rings are noncommutative analogs of 
commutative Gorenstein rings \cite{FGR75}.
Let $R$ be a two-sided Noetherian ring
and let
\[
  0\lto R \lto  I^0(R)\lto I^1(R) \lto \cdots
\]
be a minimal injective coresolution of $R$. 
The ring $R$ is called $n$-Gorenstein if
the flat dimension of  $I^g(R)$ is at most $g$ for every $0\leq g<n$, 
and it is called Auslander-Gorenstein 
if these inequalities hold in all degrees and 
the two-sided selfinjective dimensions are finite.
An Auslander-Gorenstein ring $R$ is said to be Auslander regular 
provided that the global dimension $R$ is finite.
Auslander regular rings generalize the class of  commutative regular rings.

Well-known families of Auslander-Gorenstein rings are Weyl algebras \cite{Vo00}, 
enveloping algebras of finite dimensional Lie algebras \cite{Vo00},
higher Auslander algebras \cite{Iya07}, 
and blocks of the BGG category $\scO$ \cite{KMM21}.

After establishing a criterion for two-sided Noetherian rings, 
we focus on its applications to Artin algebras, 
particularly Nakayama algebras.
Determining whether an Artin algebra is Auslander-Gorenstein appears 
to be a challenging problem.
Several recent developments address the problem of deciding 
when an Artin algebra is Auslander-Gorenstein.
For instance, it is proved in \cite{IM22} that the incidence algebra $kL$ 
of a finite lattice $L$ 
over a field $k$ is Auslander-Gorenstein if and only if $L$ is distributive.
In \cite{Kla25} it is shown that a gentle algebra $kQ/I$
is Auslander-Gorenstein  precisely when, 
for each vertex $a$ of  $Q$, 
the in-degree of $a$ is two  if and only if the out-degree of $a$ is also two.
Furthermore, the problem of classifying all Auslander-Gorenstein monomial algebras 
is reduced to that of classifying all Auslander-Gorenstein Nakayama algebras.

An important feature of Auslander-Gorenstein algebras is the
Auslander-Reiten bijection.
In \cite{AR94}, Auslander and Reiten showed that an Auslander-Gorenstein
algebra admits a natural bijection between the isomorphism
classes of indecomposable injective modules and the isomorphism
classes of indecomposable projective modules.
This bijection is now known as the Auslander-Reiten bijection.

Recent work has revealed several further interpretations of this
bijection. Kl\'asz, Marczinzik and Thomas proved that the
Auslander-Reiten bijection agrees, at the level of simple modules, with
Iyama's grade bijection \cite{KMT25}. In the Auslander
regular case, they further showed that the corresponding permutation is
the permutation arising from the Bruhat decomposition of the Coxeter
matrix. Combined with the description of Auslander regular incidence
algebras obtained by Iyama and Marczinzik
\cite{IM22}, this identifies the
Auslander-Reiten permutation of the incidence algebra of a
distributive lattice with the rowmotion permutation of the lattice.
Kl\'asz proved that a monomial algebra is
Auslander-Gorenstein if and only if its Auslander-Reiten map is
well-defined and bijective \cite{Kla25}, thereby confirming a
conjecture of Marczinzik for monomial algebras. These developments
suggest that the Auslander-Reiten bijection is not merely a consequence
of the Auslander-Gorenstein property, but may also serve as an effective
tool for detecting this property.

Although the Auslander-Reiten bijection is defined in terms of the
projective resolutions of indecomposable injective modules, 
the degrees at which these modules first appear in the minimal injective
coresolution completely determine the $n$-Gorenstein property.

Our first aim is to formulate this observation as a Serre-type condition.
Let $R$ be a two-sided Noetherian ring.
For an indecomposable injective $R$-module $I$, 
let $\occ_R(I)$ be the occurrence degree 
at which $I$ first appears in the minimal injective coresolution of $R$.  
For a nonnegative integer $n$, we introduce the condition
\[
  \occ_R(I)\geq \min\{n,\fdim_R I\}
\]
for every indecomposable injective $I$, and prove that it is equivalent to
$R$ being $n$-Gorenstein.  
In the commutative Noetherian case this is
compatible with the local characterization of Fossum and Reiten
\cite{RF72}.  
For Artin algebras it becomes a criterion involving the grade
of each simple module and the projective dimension of its injective
envelope.

The second aim is to apply this criterion to 
Nakayama algebras by means of syzygy filtration.
The construction was introduced by Sen for cyclic Nakayama
algebras in \cite{Sen19} and was extended to general Nakayama
algebras in \cite{She26}. 
It associates with a Nakayama algebra $A$ a
smaller Nakayama algebra $\varepsilon(A)$.
For any $n\geq 2$, we prove the  reduction formula
\[
  A\text{ is }n \text{-Gorenstein} 
  \iff
  A\text{ is }2 \text{-Gorenstein and } \varepsilon(A)
  \text{ is }(n-2) \text{-Gorenstein}.
\]
It follows that syzygy filtration 
preserves the Auslander-Gorenstein property and reflects it 
for $2$-Gorenstein Nakayama algebras. 
In the Auslander-Gorenstein case, the selfinjective dimension decreases by two whenever it is at least two.

Finally, let $A$ be an  Auslander-Gorenstein Nakayama algebra.  
We prove that, for every simple module $S$ and every positive odd integer $n$, 
the $n$-th Ext module  of $S$ with coefficients in $A$ is either zero or a 
simple right $A$-module. 
For basic Nakayama algebras, 
this implies the corresponding multiplicity-free statement
for the odd terms in a minimal injective coresolution of $A$. 
For split basic Nakayama algebras over a field, 
this establishes the dimension bound conjectured 
by  Kl\'asz, Kleinau and Marczinzik \cite{KKM26}.

The paper is organized as follows.
Section~2 establishes relationships among grades, occurrence degrees,
and flat dimensions of modules. 
Section~3 introduces the
Serre-type condition and characterizes $n$-Gorenstein rings in terms
of occurrence degrees and  flat dimensions. 
In Section~4, we study the behavior of
Auslander-Gorenstein Nakayama algebras under syzygy filtration.
Finally, Section~5 applies these results to the structure of odd Ext
modules and obtains a module-theoretic strengthening of the dimension
bound conjecture.

Throughout this paper, all modules are left modules unless otherwise stated.

\section{Grades, occurrence degrees and flat dimensions}
\label{sec:grade-occ-fdim}

In this section, we study the relationship between 
the occurrence degree and flat dimension 
of indecomposable injective modules over Noetherian rings.
We refer the reader  to  \cite{Lam99} for foundational background  
on rings and their modules.

Let $R$ be a ring. 
For an $R$-module $M$, 
a submodule $N$ of $M$ is called essential 
if the intersection $N\cap L$ is nonzero for every nonzero submodule $L$ of $M$.
The intersection of all essential submodules of $M$ is precisely 
the socle of $M$, that is, the sum of all simple submodules of $M$.

For every $R$-module $M$, an injective envelope  $I(M)$ of $M$ 
is an injective $R$-module containing $M$ as an essential submodule.
The injective envelope of every $R$-module $M$ always exists and 
any two injective envelopes of $M$ are isomorphic over $M$.

Let $R$ be a left Noetherian ring. 
For a finitely generated $R$-module $X$, 
the grade of $X$ is given by
\[
    \grade_RX=\min\{n\ge 0\mid \Ext^n_R(X,R)\neq 0\}.
\]
The minimum of the empty set is understood to be $\infty$.
Throughout this paper, we use the conventions
\[
n<\infty,\text{ and } \infty \pm n=\infty
\]
for every integer $n\geq 0$.

The following elementary observation will be used repeatedly.

\begin{lem}\label{lem:grade-flat}
Let $X$ be a finitely generated  $R$-module and let $M$ be an $R$-module. 
For $n\geq 0$, if
\[
\Ext^n_R(X,M)\neq 0,
\]
then
\[
\grade_R X\leq n+\fdim_R M.
\]
\end{lem}

\begin{proof}
The assertion is trivial when $\fdim_R M=\infty$. 
Suppose that $d=\fdim_R M$ is finite 
and suppose, to the contrary, that $\grade_R X>n+d$. Then
\[
\Ext_R^i(X,R)=0,
\]
for  every $0\leq i\leq n+d$.
Since $R$ is left Noetherian and $X$ is finitely generated, 
the module $X$ admits a projective resolution by 
finitely generated projective modules. 
Hence, $\Ext_R^i(X,-)$ commutes with filtered direct limits. 
By the Govorov-Lazard theorem  \cite{Laz69}, 
every flat  $R$-module is a filtered direct limit of 
finitely generated free modules. It follows  that
\begin{equation}\label{eq:vanish-flat}
\Ext_R^i(X,F)=0,
\end{equation}
for  every $0\leq i\leq n+d$ and every flat $R$-module $F$.

If $d=0$, then $M$ is flat, and \eqref{eq:vanish-flat} immediately gives that
$\Ext_R^n(X,M)$ is zero,
a contradiction. 
Assume now that $d\geq 1$, and choose a flat resolution
\[
0\lto F_d\overset{\partial^d}\lto F_{d-1}\overset{\partial^{d-1}}\lto\cdots
\lto F_0\overset{\partial^{0}}\lto M\lto 0
\]
where every $F_j$ is flat.
For $0\leq j\leq d$, let $K_j$ be the image of $\partial^j$. 
Then $K_d$ is isomorphic to $F_d$.
The flat resolution 
and \eqref{eq:vanish-flat} give an isomorphism
\[
\Ext_R^{n+j}(X,K_j)\cong \Ext_R^{n+j+1}(X,K_{j+1})
\]
by dimension shifting. Consequently,
\[
\Ext_R^n(X,M)
\cong \Ext_R^{n+1}(X,K_1)
\cong\cdots\cong
\Ext_R^{n+d}(X,K_d)=0,
\]
again a contradiction.
\end{proof}

Let $R$ be a left Noetherian ring. 
Consider a minimal injective coresolution 
\[
    0 \to R\to I^0(R)\overset{f^0}\lto I^1(R)\overset{f^1}\lto \cdots
\]
of $R$, where $\Ker f^n$ is an essential submodule of $I^n(R)$ for every $n\geq 0$.
For an indecomposable injective $R$-module $I$, 
its occurrence degree is defined  by
\[
\occ_R(I)=
\min\{\,i\geq 0\mid I\text{ is a direct summand of }I^i(R)\,\}.
\]
The minimum of the empty set is understood to be $\infty$.

The following lemma shows that the occurrence degree 
of an injective envelope of a simple module is determined by its grade.

\begin{lem} \label{lem:simple-grade}
Let $S$ be a simple $R$-module. Then
\[
    \occ_RI(S)=\grade_RS.
\]
\end{lem}

\begin{proof}
Take a minimal injective coresolution of $R$ as follows
\[
    0 \to R\to  I^0(R)\overset{f^0}\lto I^1(R)\overset{f^1}\lto \cdots.
\]
For every $n\ge 0$, $\Ker f^n$ is an essential submodule of $I^n(R)$.
Since $S$ is simple, $\Hom_R(S,f^n)$ is zero.
By applying $\Hom_R(S,-)$ to the minimal injective coresolution, 
there is an isomorphism
\[
    \Ext^n_R(S,R)\cong \Hom_R(S,I^n(R)).
\]

Suppose that there is a nonzero homomorphism
\[f\colon S\to I^n(R).\] 
Since $S$ is simple, the homomorphism $f$ is injective. 
By the injectivity of $I^n(R)$, the map $f$ extends to a homomorphism
\[
\widetilde f\colon I(S)\longrightarrow I^n(R).
\]
Because $S$ is an essential submodule of $I(S)$, 
the homomorphism $\widetilde f$ is injective.
Since $I(S)$ is injective, 
this embedding splits and $I(S)$ is a direct summand of $I^n(R)$. 
It follows that $\Ext^n_R(S,R)$ is nonzero if and only if 
$I(S)$ is a  direct summand of  $I^n(R)$.
Therefore, $\occ_RI(S)$ and $\grade_RS$ are equal.
\end{proof}

Let $k$ be a commutative Noetherian ring and 
let $R$ be a module finite $k$-algebra.
For a prime two-sided  ideal $P$ of $R$, 
by \cite[V.4]{Gab62} 
there exists an indecomposable injective $R$-module $I_P$
such that $I(R/P)$ is a finite direct sum of copies of $I_P$.
This gives a bijection between
the prime  ideals of $R$ and the isomorphism classes of 
indecomposable injective $R$-modules.

Let $P$ be a prime ideal of $R$ and put $\fkp=P\cap k$.
Then $\fkp$ is a prime ideal of $k$.
For an $R$-module $M$,  
the localization $M_\fkp$ at $\fkp$ is an $R_\fkp$-module.
Localizing at $\fkp$ is an exact functor which preserves injective envelopes.

We need the following; see \cite[Lemma 5.1]{Skr17} and \cite[Corollary 1.3]{Bass62}.

\begin{lem} \label{lem:noeth-algebra}
Let $P$ be a prime  ideal of $R$ and put
$\fkp=P\cap k$. Then the following statements hold.
\begin{enumerate}
    \item $I(R/P)_\fkp\cong I(R_\fkp/PR_\fkp)$;
    \item $R_\fkp/PR_\fkp$ is a simple Artinian ring;
    \item $(I_P)_\fkp$ is an indecomposable injective $R_\fkp$-module;
    \item $\occ_RI_P=\occ_{R_\fkp}(I_P)_\fkp$.
\end{enumerate}
\end{lem}

We have the following comparison result.

\begin{thm} \label{thm:occ-fdim}
    Let $R$ be a module finite algebra over a commutative Noetherian ring $k$.
    Then the inequality
    \[\occ_RI\leq \fdim_RI\]
    holds for every indecomposable injective $R$-module $I$.
\end{thm}

\begin{proof}
Let $I$ be an indecomposable injective $R$-module.
Then there is a prime ideal $P$ of $R$ such that $I$ 
is isomorphic to $I_P$.
Put $\fkp=P\cap k$. 
By Lemma \ref{lem:noeth-algebra}, 
$R_\fkp/PR_\fkp$ is a simple Artinian ring.
There exists a simple $R_\fkp$-module $S_P$ such that $R_\fkp/PR_\fkp$
is a finite direct sum of copies of $S_P$.
Since $I(R/P)_\fkp$ and $I(R_\fkp/PR_\fkp)$ are isomorphic,
it follows that $(I_P)_\fkp$ is isomorphic to 
an injective envelope $I(S_P)$ of $S_P$.

By Lemmas \ref{lem:simple-grade} and \ref{lem:noeth-algebra}, we have
\[
\occ_RI_P=\occ_{R_\fkp}(I_P)_\fkp=\grade_{R_\fkp}S_P. 
\]
Lemma \ref{lem:grade-flat} gives  an inequality
\[
\grade_{R_\fkp}S_P\leq \fdim_{R_\fkp}(I_P)_\fkp.
\]
Localizing a flat resolution
of $I_P$ gives a flat resolution of $(I_P)_\fkp$ over $R_\fkp$.
Then 
\[\fdim_{R_\fkp}(I_P)_\fkp\leq \fdim_{R}I_P.\] 
Combining the  displayed formulas gives the inequality
\[\occ_RI\leq \fdim_RI\]
as required.
\end{proof}

\section{A Serre-type criterion for Gorenstein rings}
\label{sec:serre-condition}
In this section, we introduce a  Serre-type condition for 
two-sided Noetherian rings, 
and prove that a two-sided Noetherian ring is $n$-Gorenstein 
if and only if this condition holds.

Given an integer $n \geq 0$, 
a two-sided Noetherian ring  $R$ is called {$n$-Gorenstein} 
if the flat dimension of $I^g(R)$ is at most $g$ for all $g<n$, 
and $R$ is called {Auslander-Gorenstein} 
if $R$ is $n$-Gorenstein for all $n$ and 
the two-sided selfinjective dimensions of $R$  are finite. 
Furthermore, $R$ is called {Auslander regular} 
if $R$ is Auslander-Gorenstein and has finite global dimension.

The $n$-Gorenstein property is left-right symmetric; 
see \cite[Theorem 3.7]{FGR75}.

\begin{lem} \label{lem:left-right}
For any $n\ge 0$, $R$ is $n$-Gorenstein if and only if 
$R^{\op}$ is $n$-Gorenstein. 
\end{lem}

Let $R$ be a commutative Noetherian local ring 
and $\fkm$ the maximal ideal of $R$.
Recall that $R$ is called Gorenstein provided that 
its selfinjective dimension  is finite.
The grade of the $R$-module $R/\fkm$ is called the depth of $R$.
It is known that $R$ is Gorenstein if 
and only if the depth and selfinjective dimension of $R$ are equal \cite{BH98}.

We recall the following 
characterization of commutative $n$-Gorenstein rings.

\begin{lem} 
    Let $R$ be a commutative Noetherian ring and $n \ge 0$. 
    The following are equivalent.
    \begin{enumerate}
        \item $R$ is $n$-Gorenstein;
        \item $\depth R_\fkp\ge \min\{n,\fdim_{R}I(R/\fkp)\}$ 
        for every prime ideal $\fkp$ of $R$;
        \item  $R_\fkp$ is Gorenstein for every prime ideal 
        $\fkp$ of $R$ with $\depth R_\fkp<n$.
    \end{enumerate}
\end{lem}

\begin{proof}
(1) $\iff$ (3) 
The equivalence is established in \cite{RF72}.

(2) $\iff$ (3)
Since $R$ is a commutative Noetherian ring, we have
\[
    \idim_{R_{\fkp}}R_{\fkp}
    =\fdim_{R_\fkp}I(R_\fkp/\fkp R_\fkp).
\]
By Lemma \ref{lem:grade-flat} 
and the local characterization of flat dimension,
\[
    \fdim_RI(R/\fkp)=\fdim_{R_\fkp}I(R_\fkp/\fkp R_\fkp)
    \ge\depth R_\fkp.
\]
The equality holds if and only if $R_{\fkp}$ is Gorenstein.
These two conditions are therefore equivalent.
\end{proof}

Inspired by the above lemma, 
we introduce the following Serre-type condition 
for two-sided Noetherian rings.

\begin{defn}\label{def:serre-condition}
    Let $R$ be a two-sided Noetherian ring and $n \ge 0$.
    We say that $R$  satisfies {$(G_n)$} provided that
    \[
        \occ_R(I)\ge \min\{n,\fdim_RI\}
    \] 
    for every indecomposable injective $R$-module $I$. 
\end{defn}

\begin{rmk}
    If $R$ is a module finite algebra over a commutative Noetherian ring,
    Theorem \ref{thm:occ-fdim}  implies that the inequality 
    \[
        \fdim_RI\ge \occ_RI
    \] 
    holds for every indecomposable injective $R$-module $I$.
    It is unknown whether the inequality holds 
    for arbitrary two-sided Noetherian rings.
    The notation $(G_n)$ used here should not be confused with other
    conditions denoted by $(G_n)$ in commutative algebra; 
    throughout this paper, 
    it refers exclusively to Definition~\ref{def:serre-condition}.
\end{rmk}

The following is the main result of this section.
It gives an alternative characterization of $n$-Gorenstein rings.

\begin{thm} \label{thm:n-gor-ring}
    Let  $R$ be a two-sided Noetherian ring and $n \ge 0$.
    Then $R$ is $n$-Gorenstein if and only if $R$ satisfies $(G_n)$.
\end{thm}

\begin{proof}
    $``\implies"$
    Let $I$ be an indecomposable injective $R$-module.
    If $\occ_RI\ge n$, then the desired inequality is immediate. 
    If $\occ_RI=g< n$, then $I$ is a direct summand of $I^g(R)$, 
    and thus
    \[
        \fdim_RI\leq \fdim_RI^g(R)\leq g=\occ_R(I).
    \]
    Then $R$ satisfies $(G_n)$.

    $``\impliedby"$
    Since $R$ is left Noetherian,
    by \cite{FW67} every injective $R$-module  is a direct sum of 
    indecomposable injective $R$-modules.
    Fix $g<n$ and let $I$ be an indecomposable direct summand  of $I^g(R)$.
    Then $\occ_R(I)\le g<n$. 
    The inequality
    \[
        \occ_RI\geq \min\{n,\fdim_RI\}
    \]
    forces $\fdim_RI< n$ and hence $\fdim_RI\le g$.
    Since direct sums of modules of flat dimension at most 
    $g$ again have flat dimension at most $g$,
    we have $\fdim_RI^g(R)\le g$.
    This holds for every $g<n$.
    It follows that $R$ is $n$-Gorenstein.
\end{proof}

\begin{prop} \label{prop:all-gor-ring}
    Let  $R$ be a two-sided Noetherian ring.  
    The following statements are equivalent.
    \begin{enumerate}
        \item $R$ is $n$-Gorenstein for all $n$;
        \item $\fdim_RI\leq \occ_RI$ 
        for every indecomposable injective $R$-module $I$.
    \end{enumerate}
\end{prop}

\begin{proof}
The equivalence follows from Theorem \ref{thm:n-gor-ring}.
\end{proof}

We now focus our attention on Artin algebras.
Recall that an Artin algebra is an algebra 
that is  module finite over a commutative Artinian ring.
We refer to \cite{ARS95} for basic background  on Artin algebras.

Let $k$ be a commutative Artinian ring and $E$ be a minimal 
injective cogenerator of $k$-modules.
Denote by $D=\Hom_k(-,E)$ the standard duality.

Let $A$ be an Artin algebra over $k$.
Then every indecomposable injective $A$-module is 
isomorphic to the injective envelope 
$I(S)$ of a simple $A$-module $S$.
By Lemmas \ref{lem:simple-grade} and \ref{lem:grade-flat},
it follows that
\[
    \occ_AI(S)=\grade_AS \leq \pdim_AI(S).
\]

By Theorem \ref{thm:n-gor-ring}, 
we have the following criterion  characterizing the $n$-Gorenstein property 
for Artin algebras.

\begin{thm}  \label{thm:n-gor-algebra}
    Let  $A$ be an Artin algebra and $n \ge 0$.
    The following  are equivalent.
    \begin{enumerate}
    \item $A$ is $n$-Gorenstein;
    \item $\grade_AS\geq\min\{n,\pdim_AI(S)\}$ for each simple $A$-module $S$;
    \item $\pdim_AI(S)=\grade_AS$ for each simple $A$-module with $\grade_AS<n$.
\end{enumerate}  
\end{thm}

By Proposition \ref{prop:all-gor-ring}, we have the following.

\begin{prop} \label{prop:all-gor-algebra}
    Let $A$ be an Artin algebra. 
    The following  are  equivalent.
    \begin{enumerate}
        \item $A$ is $n$-Gorenstein for all $n$;
        \item $\pdim_AI(S)=\grade_AS$ for every simple $A$-module $S$.
    \end{enumerate}
\end{prop}

Recall the following  result; see \cite[Corollary 5.5]{AR94}.

\begin{lem} \label{lem:gor}
    Let $A$ be an Artin algebra which  is $n$-Gorenstein for all $n$.
    Then the left selfinjective dimension of $A$ is finite if and only if 
    the right selfinjective dimension of $A$ is finite.
\end{lem}

This gives a new proof of  \cite[Theorem 3.3]{KMT25}.
We note that Theorem \ref{thm:n-gor-algebra} 
can be regarded as a refinement of this result.

\begin{cor} \label{cor:aus-gor-algebra}
    Let $A$ be an Artin algebra. The following are equivalent.
    \begin{enumerate}
        \item $A$ is Auslander-Gorenstein;
        \item $\pdim_AI(S)=\grade_AS<\infty$ for every simple $A$-module $S$.
    \end{enumerate}
\end{cor}

\section{The Auslander-Gorenstein property for Nakayama algebras}
\label{sec:aus-gor-nak}

This section is devoted to the study of the Auslander-Gorenstein 
property for Nakayama algebras.
We prove that the syzygy filtered algebra 
of every Auslander-Gorenstein Nakayama algebra is again Auslander-Gorenstein.

The results in this section provide the homological framework 
needed in the final section. 
In particular, the behavior of the Auslander-Gorenstein property 
under syzygy filtration 
will allow us to analyze the structure of Ext-modules of simple modules.

A Nakayama algebra is an Artin algebra 
whose indecomposable modules are all uniserial.
The quiver of a connected Nakayama algebra is either linear or cyclic.

Throughout Sections \ref{sec:aus-gor-nak} and \ref{sec:odd-ext}, 
we work with basic Nakayama algebras $A$ satisfying the following two properties:
\begin{enumerate}
  \item[(N1)] either the  quiver of $A$ is cyclic, 
  or every connected component of the quiver of  $A$ is linear;
  \item[(N2)] the endomorphism algebras of simple $A$-modules are mutually isomorphic.
\end{enumerate}

\begin{rmk}
Although \cite{She26} is stated for finite dimensional basic
algebras over an algebraically closed field, the proofs of
the results used below depend only on the uniseriality of
indecomposable modules and on the fact 
that the endomorphism rings of the simple modules are mutually
isomorphic. Hence, they remain valid for basic Artin
Nakayama algebras satisfying (N1) and (N2).
\end{rmk}

It is well known that a Nakayama algebra has dominant dimension at least $1$;
or equivalently, 
the injective envelope of its regular module is projective; see \cite{ARS95}. 
Thus, every Nakayama algebra is $1$-Gorenstein.

For any Nakayama algebra $A$, 
we  fix a cyclic translation $\tau$ on the complete set of simple $A$-modules. 
By construction,  $\tau$ agrees with the Auslander-Reiten translation 
for non-projective simple modules and 
sends each projective simple module to an injective simple module. 
Such a cyclic translation is uniquely determined by $A$, 
provided that the algebra $A$ is connected.

Let $e_1,e_2,\cdots,e_m$ be a complete set of 
pairwise orthogonal primitive idempotents in $A$.
For $1\leq i\leq m$, write $P(i)$ for $Ae_i$ and $S(i)$ for $Ae_i/\rad Ae_i$.
Then 
\[\scS=\{S(i)\mid 1\leq i\leq m\}\]
is a complete set of simple $A$-modules.
The simple $A$-modules ordered such that
\[
    \tau(S(i))=S(i+1) \text{ for all } 1\leq i<m \text{ and }  \tau S(m)=S(1).
\] 
For each index $i$, let $c_i$ denote the composition length of $P(i)$.
The tuple 
\[(c_1,c_2,\ldots,c_m)\]
is called a Kupisch series of $A$.
Note that the inequality
\[c_i\leq c_{i+1}+1\]
holds for every index $i$. 
We adopt the convention that $c_{m+1}=c_1$.

For $1\leq i\leq m$ and $1\leq \ell\leq c_i$, 
we write the interval 
\[[i,i+\ell)\] 
for the indecomposable module $P(i)/\rad^{\ell}P(i)$.

In \cite{Rin13}, 
Ringel introduced a map $\gamma$ on the set of simple $A$-modules, 
defined by
\[ 
    \gamma(S)=\tau\soc P(S).
\]

We will use the following result; see  \cite[A.6]{Rin21}.

\begin{lem}\label{lem:bijection}
Let $A$ be a Nakayama algebra and $S$ a simple $A$-module.
\begin{enumerate}
    \item $\pdim_AS\neq 1$ if and only if $P(\tau S)$ is injective;
    \item $\idim_AS\neq 1$ if and only if $I(\tau^{-1} S)$ is projective;
    \item $S$ lies in the image of $\gamma$ if and only if  $\idim_AS\neq 1$;  
    \item $\gamma$ restricts to a  bijection 
\[\{S\in\scS\mid \pdim_AS\neq 1\}\overset{\sim}\lto \{T\in\scS\mid \idim_AT\neq 1\}.\]
\end{enumerate}
\end{lem}

Let $A$ be a Nakayama algebra. 
Fix a simple $A$-module  $S$ with $\pdim_AS\neq 1$.
Lemma \ref{lem:bijection} implies that 
the injective dimension of $\gamma S$ is different from $1$.
By \cite{Rin21,Sen19,She26}, 
the base module $\nabla(S)$ is defined by 
the exact sequence of $A$-modules
\[
    0\to \nabla(S)\to P(\tau S)\to P(S)\to S\to 0.
\]
By Lemma \ref{lem:bijection}, 
the module $P(\tau S)$ is  injective. 
The top  of $\nabla(S)$ is isomorphic to $\gamma S$
and the injective dimension of $\gamma S$ is different from $1$.
The base module $\nabla(S)$ is the minimal submodule of $P(\tau{S})$ 
whose top  has injective dimension different from $1$.

The set of all base modules
\[
    \scB=\{\nabla(S)\mid \pdim_AS\neq 1\}
\]
forms a semibrick in the category $\mod A$ of all finitely generated $A$-modules.
The filtration category $\scE$ of $\scB$ is 
an exact abelian subcategory of $\mod A$ whose embedding preserves projective covers.
Moreover, $\scE$ contains the second  syzygy 
in a minimal projective resolution of any finitely generated $A$-module.

Denote by $\scP$ the set of projective covers of all base modules.
The additive closure  of $\scP$ coincides 
exactly with the projective objects of $\scE $.

\begin{lem}\label{lem:ext-comparison}
For all $X,Y\in\scE$ and all $n\geq 0$, 
there is a natural isomorphism
\[
  \Ext^n_{\scE}(X,Y)\cong \Ext^n_A(X,Y).
\]
Moreover, $\pdim_\scE X=\pdim_A X$ for all $X\in\scE$.
\end{lem}

\begin{proof}
The filtration category $\scE$ has enough projective objects, 
and its projective objects are precisely
the modules in $\add A_\varepsilon$. 
In particular, every projective object of 
$\scE$ is projective as an $A$-module. 
Hence, every projective resolution of 
$X$ in $\scE$ is also a projective resolution in $\mod A$.
Since $\scE$ is a full subcategory
of $\mod A$, applying $\Hom(-,Y)$ to this resolution computes 
both Ext modules and gives the asserted natural isomorphism.

Since $\scE$ is a full exact abelian subcategory of
$\mod A$, and the embedding preserves projective
covers, the kernel of the projective cover of an object of
$\scE$ is computed in $\mod A$ and again
belongs to $\scE$. Inductively, the minimal projective
resolution in $\scE$ agrees with the corresponding
minimal projective resolution in $\mod A$.
Then $\pdim_\scE X$ and $\pdim_A X$ are equal for all $X$ in $\scE$.
\end{proof}

Denote by $A_\varepsilon$ the direct sum of all modules in $\scP$.
The syzygy filtered algebra $\varepsilon(A)$ of $A$ is 
the opposite endomorphism algebra of $A_\varepsilon$, that is, 
\[
    \varepsilon(A)=\End_A^{\op}(A_\varepsilon)=
    \End_A^{\op}(\bigoplus\nolimits_{P\in\scP} P).
\]
It follows that the filtration category $\scE $ is equivalent to the category 
$\mod\varepsilon(A)$ of all finitely generated $\varepsilon(A)$-modules.
The algebra $\varepsilon(A)$ is a Nakayama algebra 
of rank  at most the rank of $A$.
The equality holds if and only if $A$ is selfinjective.

The construction of base modules yields a bijection 
from the simple $A$-modules with projective dimension 
different from $1$ 
onto  the  simple objects of $\scE$. 
This induces a cyclic translation $\tau_\scE $ 
on the set of simple objects of $\scE$ such that
\[
\tau_\scE \nabla(S)=\nabla(S')\iff \gamma\tau (S)=\gamma(S').
\]
This  gives rise to a reduction
\[(A,\tau)\to (\varepsilon(A),\tau_\scE).\]
The Kupisch series of $\varepsilon(A)$ 
can be explicitly computed from that of $A$.

Syzygy filtration provides a useful reduction technique 
for studying Nakayama algebras. 
It interacts particularly well with several homological dimensions, 
including global dimension, self-injective dimension, finitistic dimension, 
$\varphi$-dimension, and dominant dimension. 
We refer the reader to \cite{Sen19} for further details.

\begin{exam}
Let $A$ be a Nakayama algebra with Kupisch series
\[ 
(3,3,3,3,3,2,1). 
\]
It admits five base modules as follows
\[
\nabla(1)=[4,5),\;\nabla(2)=[5,6),\;\nabla(3)=[6,7),\;
\nabla(4)=[7,8),\;\nabla(7)=[1,4).
\]
A direct computation shows that the Kupisch series of $\varepsilon(A)$ is given by
\[(1,3,3,2,1).\]
Note that $\varepsilon(A)$ is a direct sum of two linear Nakayama algebras.
\end{exam}

We need the following; see \cite{Rin21}.

\begin{lem} \label{lem:madsen}
Let $M$ be an indecomposable $A$-module.
\begin{enumerate}
    \item $\pdim_AM=1$ if and only if 
    every composition factor of $M$ has projective dimension $1$;
    \item $\idim_AM=1$ if and only if 
    every composition factor of $M$ has injective dimension $1$.
\end{enumerate}
\end{lem}

Let $S$ be a  simple module with $\pdim_AS \neq 1$.
By Lemma \ref{lem:madsen}, $\pdim_AI(S)\neq 1$.
There exists a commutative diagram with exact rows as depicted below.
\begin{equation} \label{eq:comdiag}
\begin{tikzcd}
0\arrow[r] &\nabla(S)\arrow[r,]\arrow[d,] &P(\tau{S})\arrow[r]\arrow[d,"="] 
&P(S) \arrow[r]\arrow[d] & S \arrow[r]\arrow[d] &0\\
0\arrow[r] &J(S) \arrow[r] &P(\tau{S})\arrow[r] &P(I(S)) \arrow[r] &I(S)\arrow[r] &0
\end{tikzcd}
\end{equation}

The module $J(S)$ is the maximal submodule of $P(\tau{S})$ 
whose top has injective dimension different from $1$.
If $I(S)$ is a projective module, then $J(S)$ is a projective and injective $A$-module. 
If $\pdim_A I(S)\geq2$, then $J(S)$ is a non-injective $A$-module.

By \cite[Lemma 3.3]{She26}, the module $J(S)$ is the
injective envelope of $\nabla(S)$ in the filtration category $\scE$.
Moreover, \[
    \scJ=\{J(S)\mid \pdim_AS\neq 1\}
\]
a complete set of indecomposable injective objects of $\scE$.

\begin{prop}\label{prop:injdim-J}
Let $A$ be a Nakayama algebra and $S$ be a  simple module 
such that $\pdim_AS\neq 1$. 
Then
\[
    \pdim_{\scE}J(S)=\max\{\pdim_AI(S)-2,0\}.
\]
\end{prop}

\begin{proof}
By Lemma \ref{lem:madsen},  $\pdim_AI(S)\neq 1$.
If $I(S)$ is projective, then  
$J(S)$ is  projective.
If $\pdim_AI(S)\ge 2$,  it follows from \eqref{eq:comdiag} that 
\[
    \pdim_{A}J(S)=\pdim_AI(S)-2.
\]
By Lemma \ref{lem:ext-comparison}, 
$\pdim_\scE J(S)=\pdim_AJ(S)$.
Therefore, the conclusion follows.
\end{proof}

Now we study the Auslander-Gorenstein property for Nakayama algebras.
Let us consider a minimal injective copresentation of $A$ as follows
\[
        0\to A\to I^0(A)\to I^1(A).
\]
Let $S$ be a simple $A$-module with $\pdim _AS=1$. 
By Lemma \ref{lem:bijection}, the sequence
\[
0\to P(\tau S)\to I(P(\tau S))\to I(S)
\]
is a minimal injective copresentation of $P(\tau S)$.
Since $A$ is a basic algebra, there is an isomorphism
\begin{equation} \label{eq:I1}
I^1(A)\cong \bigoplus_{\substack{S\in\scS\\ \pdim_AS= 1}}I(S).
\end{equation}

The following special property of Nakayama algebras will be useful.

\begin{lem} \label{lem:2-gor}
Let $A$ be a Nakayama algebra. The following are equivalent.
\begin{enumerate}
    \item $A$ is $2$-Gorenstein;
    \item $A$ is $3$-Gorenstein;
    \item $\pdim_AI(S)\le 1$ for every simple module $S$ with $\pdim_AS=1$;
    \item $\idim_AP(T)\le 1$ for every simple module $T$ with $\idim_AT=1$.
\end{enumerate}
\end{lem}

\begin{proof}
    (1) $\iff$ (2)
    This is a consequence of \cite[2.10]{FI93}.

    (1) $\implies$ (3)
    Let $S$ be a simple module with $\pdim_AS=1$.
    By \eqref{eq:I1}, the module $I(S)$ 
    is a direct summand of $I^1(A)$.
    Since $A$ is $2$-Gorenstein, $\pdim_AI(S)\le 1$.

    (3) $\implies$ (1)
    Let $S$ be a simple module with $\pdim_AS=1$.
    Then $\pdim_AI(S)\le 1$.
    By \eqref{eq:I1}, we have $\pdim_AI^1(A)\leq 1$.
    It follows that $A$ is $2$-Gorenstein.

    (1) $\iff$ (4) 
    This follows by applying the equivalence of (1) and (3) to $A^{\op}$
    and using Lemma~\ref{lem:left-right}.
\end{proof}

By Theorem \ref{thm:n-gor-algebra}, we have the following.

\begin{cor} \label{cor:s3}
If $A$ is a $2$-Gorenstein Nakayama algebra, then 
    \[
        \pdim_AI(S)=\grade_AS
    \]
for every simple module $S$ such that $\grade_AS<3$.
\end{cor}

By Proposition \ref{prop:injdim-J}, we have the following.

\begin{cor} \label{cor:selfinjdim}
Let $A$ be a $2$-Gorenstein Nakayama algebra.
Then the right selfinjective dimension of $A$ is finite if and only if the right
selfinjective dimension of $\varepsilon(A)$ is finite. More precisely,
\[
    \idim_{\varepsilon(A)^{\op}}\varepsilon(A)
    =\max\{\idim_{A^{\op}}A-2,0\}.
\]
\end{cor}

\begin{proof}
The $A$-module $D(A_A)$ is the direct sum of the indecomposable injective
 $A$-modules. Then
\[
  \idim_{A^{\op}}A=\pdim_A D(A_A)=\max_S\pdim_A I(S).
\]

Via the equivalence
$\scE\simeq\mod \varepsilon(A)$, 
the modules $J(S)$, with $\pdim_AS\neq1$, 
form a complete set of indecomposable injective objects.
For the omitted simple modules, the $2$-Gorenstein hypothesis gives
$\pdim_AI(S)\leq1$. Proposition~\ref{prop:injdim-J} therefore yields
\[
  \idim_{\varepsilon(A)^{\op}}\varepsilon(A)
  =\max_{\pdim_AS\neq1}\pdim_{\scE}J(S)
  =\max\{\idim_{A^{\op}}A-2,0\}.
\]
This completes the proof.
\end{proof}

\begin{exam} \label{exam:4-simple}
Let $A$ be a Nakayama algebra with Kupisch series 
\[(3,3,2,1).\]
There are exactly two simple modules with injective dimension $1$, 
namely $S(2)$ and $S(3)$.
Their projective covers are
\[P(2)=[2,5) \text{ and }  P(3)=[3,5).\]
Note that $P(2)$ is injective and $\idim_AP(3)=2$.
By Lemma \ref{lem:2-gor}, the algebra $A$ is not $2$-Gorenstein.
\end{exam}

For $A$-modules $M$ and $N$, denote by $\oHom_A(M,N)$
the quotient of $\Hom_A(M,N)$ by the submodule of
homomorphisms factoring through injective $A$-modules.

\begin{lem}\label{lem:uniserial-observation}
Let $A$ be a  Nakayama algebra and let $\nabla$ be a base module.  
Suppose that $T$ is a simple module with $\idim_AT=1$ and $\idim_AP(T)\le 1$, 
then
\[
  \oHom_A(P(T),D\Tr\nabla)=0.
\]\end{lem}

\begin{proof}
If $P(T)$ is injective, 
every morphism from $P(T)$  factors through an injective module, 
and the assertion is immediate.
Suppose that $\idim_AP(T)=1$. 
By Lemma~\ref{lem:madsen},
every composition factor of $P(T)$ has injective dimension $1$.

Write $\nabla=\nabla(S)$ where $\pdim_AS\neq 1$.
If $\nabla$ is projective, then $D \Tr\nabla$ is zero. 
Otherwise, $D \Tr\nabla$ is an indecomposable module. 
Using the standard description of the Auslander-Reiten
translate of an indecomposable module, we obtain
\[
\soc D\Tr\nabla(S)
\cong
\tau\soc\nabla(S)
\cong
\tau\soc P(\tau S)
=
\gamma(\tau S).
\]
The injective dimension of this simple module is different from $1$ 
by Lemma \ref{lem:bijection}.
If a nonzero morphism from $P(T)$ to $D\Tr\nabla$ existed, its image
would be a nonzero submodule of the uniserial module $D\Tr\nabla$ and would
therefore contain its socle. Then $\soc D\Tr\nabla$ would be a composition factor
of $P(T)$, contradicting Lemma~\ref{lem:madsen}. Hence,
\[\Hom_A(P(T),D\Tr\nabla)=0\] 
and the injectively stable quotient is zero as well.
\end{proof}

\begin{lem} \label{lem:ext-van}
Suppose $A$ is a $2$-Gorenstein Nakayama algebra.
Let $\nabla$ be a base module and let $T$ be a simple module 
with $\idim_AT=1$. Then 
\[
    \Ext_A^n(\nabla,P(T))= 0 \quad \text{for all } n\ge 1.
\]
\end{lem}

\begin{proof}
Since $A$ is $2$-Gorenstein and $\idim_AT=1$, 
by Lemma \ref{lem:2-gor}  $\idim_AP(T)\leq 1$. 
It follows that
\[
    \Ext_A^n(\nabla,P(T)) = 0 
\]
for any $n\ge 2$.
By the Auslander-Reiten formula and 
Lemma~\ref{lem:uniserial-observation}, we have
\[
    D\Ext_A^1(\nabla,P(T))\cong \oHom_A(P(T),D\Tr\nabla)=0.
\]
Since the  duality functor $D$ is faithful, 
$\Ext_A^1(\nabla,P(T))$ is zero.
This gives the desired conclusion.
\end{proof}

The tops of the indecomposable summands of $A_\varepsilon$ are precisely the
simple modules of injective dimension different from one. 
Consequently, there is a decomposition
\[
  A=A_1\oplus A_\varepsilon,
\]
where $A_1$ is the direct sum of the  modules $P(T)$
with $\idim_AT=1$.

\begin{cor} \label{cor:ext-2}
Let $A$ be a $2$-Gorenstein Nakayama algebra 
and $S$ a simple module  with  $\pdim_AS\neq 1$. 
For every $n\geq 3$, there is a natural isomorphism 
\[
    \Ext^n_A(S,A)\cong \Ext^{n-2}_\scE(\nabla(S),A_\varepsilon).
\]
\end{cor}

\begin{proof}
The defining exact sequence
\[
0\longrightarrow \nabla(S)\longrightarrow P(\tau S)
\longrightarrow P(S)\longrightarrow S\longrightarrow 0
\]
yields, by dimension shifting,
\[
\Ext_A^n(S,A)
\cong
\Ext_A^{n-2}(\nabla(S),A)
\]
for every $n\geq 3$.
Write $A=A_1\oplus A_\varepsilon$. 
By Lemma \ref{lem:ext-van},
\[
\Ext_A^{n-2}(\nabla(S),A_1)=0.
\]
Hence,
\[
\Ext_A^{n-2}(\nabla(S),A)
\cong
\Ext_A^{n-2}(\nabla(S),A_\varepsilon).
\]
Finally, Lemma \ref{lem:ext-comparison} gives
\[
\Ext_A^{n-2}(\nabla(S),A_\varepsilon)
\cong
\Ext_{\scE}^{n-2}
(\nabla(S),A_\varepsilon).
\]
All these isomorphisms are natural.
\end{proof}

We have the following.

\begin{prop}  \label{prop:grade}
Let $A$ be a $2$-Gorenstein Nakayama algebra 
and $S$ a simple module  with  $\pdim_AS\neq 1$. Then 
\[
    \grade_{\scE}\nabla(S)=\max\{0,\grade_AS-2\}.
\]
\end{prop}

\begin{proof}
If $\grade_AS<3$, by Corollary \ref{cor:s3} we have
\[
    \pdim_AI(S)=\grade_AS < 3.
\]
Proposition~\ref{prop:injdim-J} then shows that $J(S)$ is projective in $\scE$.
Since $\nabla(S)$ embeds into its injective envelope $J(S)$ and $J(S)$ is a
summand of $A_\varepsilon$, we have 
$\Hom_{\scE}(\nabla(S),A_\varepsilon)$ is nonzero.
Then $\grade_{\scE}\nabla(S)=0$.

Assume now that $g=\grade_AS\geq3$. 
By Lemma \ref{lem:grade-flat} $\pdim_AI(S)\geq3$, so
$J(S)$ is non-projective by Proposition~\ref{prop:injdim-J}. Since
$\varepsilon(A)$ is also a Nakayama algebra, it is $1$-Gorenstein. Applying
Theorem~\ref{thm:n-gor-algebra} inside $\scE$ gives
$\grade_{\scE}\nabla(S)\geq1$.
Then
\[
    \Hom_{\scE}(\nabla(S),A_\varepsilon)=0.
\]

Suppose first that $g=\infty$. For every $m\geq 1$,
Corollary~\ref{cor:ext-2} gives
\[
\Ext_{\scE}^{m}(\nabla(S),A_\varepsilon)\cong\Ext_A^{m+2}(S,A)=0.
\]
Together with $\Hom_{\scE}(\nabla(S),A_\varepsilon)=0$,
this shows that $\grade_{\scE}\nabla(S)=\infty.$

Assume now that $3\leq g<\infty$. Corollary~\ref{cor:ext-2}
gives
\[
\Ext_{\scE}^{m}(\nabla(S),A_\varepsilon)=0
\qquad\text{for }m<g-2,
\]
whereas $\Ext_{\scE}^{g-2}(\nabla(S),A_\varepsilon)$ is nonzero.
Therefore, the $\scE$-grade of $\nabla(S)$ is $g-2$.
\end{proof}

The following is the main result of this section.
We adopt the convention that every  algebra is 0-Gorenstein.

\begin{thm} \label{thm:nak-n-gor}
Let $A$ be a Nakayama algebra and $n\ge 2$. 
Then $A$ is $n$-Gorenstein if and only if 
$A$ is $2$-Gorenstein and $\varepsilon(A)$ is $(n-2)$-Gorenstein.
\end{thm}

\begin{proof}
$``\implies"$
Since $A$ is $n$-Gorenstein, 
$A$ is $2$-Gorenstein and 
\[
    \grade_AS\ge\min\{n,\pdim_AI(S)\}
\]
for any simple module $S$.
If $\pdim_AS\neq 1$,
by Propositions \ref{prop:grade} and  \ref{prop:injdim-J} we have
\[
    \begin{aligned}
    \grade_{\scE}\nabla(S) &=\max\{\grade_AS-2,0\}\\
                        &\ge \max\{\min\{n-2,\pdim_AI(S)-2\},0\}\\
                        & =\min\{n-2,\max\{\pdim_AI(S)-2,0\}\}\\
                        & =\min\{n-2,\pdim_{\scE}J(S)\}.
    \end{aligned}
\]
Here, we make use of the distributive identity for min and max operations.
Then the algebra $\varepsilon(A)$ is $(n-2)$-Gorenstein.

$``\impliedby"$
Let $S$ be a simple $A$-module with $\grade_AS< 2$. 
Since $A$ is $2$-Gorenstein, 
it follows from Theorem \ref{thm:n-gor-algebra} that
\[
    \grade_AS = \pdim_AI(S)\geq \min\{n,\pdim_AI(S)\}.
\]

Let $S$ be a simple $A$-module  with $\grade_AS \ge 2$.
Applying Lemma \ref{lem:grade-flat} to both $S$ and $I(S)$,
the two modules have projective dimension  at least $2$.
Since $\varepsilon(A)$ is $(n-2)$-Gorenstein, 
by Propositions \ref{prop:grade} and  \ref{prop:injdim-J}, 
we obtain
\[
    \begin{aligned} 
    \grade_AS &= \grade_{\scE}\nabla(S)+2\\
              &\ge\min\{n,\pdim_{\scE}J(S)+2\}\\
              &=\min\{n,\pdim_AI(S)\}.
    \end{aligned}
\]
Then the algebra $A$ is $n$-Gorenstein.
\end{proof}

Set $\varepsilon^0(A)=A$ and
$\varepsilon^{i}(A)=
\varepsilon(\varepsilon^{i-1}(A))$
for every $i\geq1$.
The algebra $\varepsilon^i(A)$ is said to be the 
$i$-th iterated syzygy filtered algebra of $A$.

Combining Theorem \ref{thm:nak-n-gor} with Lemma \ref{lem:2-gor}, 
we have the following.

\begin{cor} \label{cor:n-gor}
Let $A$ be a Nakayama algebra and $n\ge 0$. 
Then $A$ is $n$-Gorenstein if and only if 
$\varepsilon^i(A)$ is $2$-Gorenstein for every $0\le i < \lfloor n/2\rfloor$.
\end{cor}

\begin{proof}
For $n=0,1$ the assertion follows from the convention for $0$-Gorenstein algebras
and the fact that every Nakayama algebra is $1$-Gorenstein. For $n\geq2$, iterate
Theorem~\ref{thm:nak-n-gor}. Lemma~\ref{lem:2-gor}, which identifies the
$2$-Gorenstein and $3$-Gorenstein conditions, handles the final odd step.
\end{proof}

Combining Theorem \ref{thm:nak-n-gor} with  Corollary \ref{cor:selfinjdim}, 
we obtain the following.

\begin{thm} \label{thm:nak-ausgor}
Let $A$ be a Nakayama algebra.
Then $A$ is Auslander-Gorenstein  if and only if 
$A$ is $2$-Gorenstein and $\varepsilon(A)$ is Auslander-Gorenstein.
In this case, writing $\operatorname{vdim}A$ for the common two-sided
selfinjective dimension, one has
\[
    \operatorname{vdim}\varepsilon(A)=\max\{\operatorname{vdim}A-2,0\}.
\]
\end{thm}

\begin{proof}
By Theorem~\ref{thm:nak-n-gor}, 
the algebra $A$ is $n$-Gorenstein for every $n$ if and only if
$A$ is $2$-Gorenstein and $\varepsilon(A)$ is $n$-Gorenstein for every $n\geq 0$.
Corollary~\ref{cor:selfinjdim} shows that, under the $2$-Gorenstein hypothesis,
the right selfinjective dimension of $A$ is finite if and only if that of
$\varepsilon(A)$ is finite. Lemma~\ref{lem:gor} then gives the corresponding
finiteness on the left. This proves the equivalence.

For a two-sided Noetherian ring whose two selfinjective dimensions are finite,
the two dimensions are equal; see \cite{Zak69}. The displayed formula therefore
follows from Corollary~\ref{cor:selfinjdim}.
\end{proof}

For algebras of selfinjective dimension at least $3$, we have the following.

\begin{cor} \label{cor:nak-d-ausgor}
Let $A$ be a Nakayama algebra and $d\geq 3$.
Then $A$ is Auslander-Gorenstein of selfinjective dimension $d$ if and only if 
$A$ is $2$-Gorenstein and $\varepsilon(A)$ is Auslander-Gorenstein 
of selfinjective dimension $d-2$.
\end{cor}

Recall that the rank of $\varepsilon(A)$ is at most the rank of $A$.
The equality holds if and only if  $A$ is selfinjective.
Then there is an integer $r$ such that 
$\varepsilon^r(A)$ is selfinjective.

\begin{cor}
A Nakayama algebra  $A$  is Auslander-Gorenstein if and only if
$\varepsilon^i(A)$ is $2$-Gorenstein for every $i\geq 0$.
\end{cor}   

\begin{proof}
Suppose first that $A$ is Auslander-Gorenstein.
By Theorem \ref{thm:nak-ausgor} each iterated syzygy filtered algebra
$\varepsilon^i(A)$ is Auslander-Gorenstein, and hence
$2$-Gorenstein.

Conversely, suppose that $\varepsilon^i(A)$ is
$2$-Gorenstein for every $i\geq0$. By Corollary \ref{cor:n-gor},
the algebra $A$ is $n$-Gorenstein for all $n\geq0$.
Choose $r$ such that $\varepsilon^r(A)$ is
selfinjective. 
Repeated application of Corollary \ref{cor:selfinjdim} shows
that the right selfinjective dimension of $A$ is finite.
Lemma \ref{lem:gor} then implies that its left selfinjective dimension
is also finite. Thus, $A$ is Auslander-Gorenstein.
\end{proof}

\begin{exam}
Let $A$ be a Nakayama algebra with Kupisch series 
\[(3,3,3,3,3,2,1).\]
There are exactly two simple modules with injective dimension $1$, 
namely $S(2)$ and $S(3)$.
Their projective covers are
\[P(2)=[2,5) \text{ and }  P(3)=[3,6).\]
Note that both $P(2)$ and $P(3)$ are injective $A$-modules.
Lemma \ref{lem:2-gor} implies that $A$ is a $2$-Gorenstein algebra.

The Kupisch series of $\varepsilon(A)$ is given by $(1,3,3,2,1)$.
Example \ref{exam:4-simple} indicates that 
$\varepsilon(A)$  fails to be $2$-Gorenstein.
It follows that $A$ is not $4$-Gorenstein.

For any $n\geq 1$, let $A_n$ be a Nakayama algebra with Kupisch series 
\[((3)^{3n-1},2,1).\]
For $n\geq2$,  the Kupisch series of  $\varepsilon(A_n)$ is given by
\[(1,(3)^{3n-4},2,1).\]
The case $n=1$ is the algebra in Example \ref{exam:4-simple} 
and is treated separately.
One can show that $A_n$ is $(2n-2)$-Gorenstein, but not $2n$-Gorenstein.
\end{exam}

\begin{exam}
Let $A$ be a Nakayama algebra with Kupisch series 
\[(3,2,2,1).\]
There are exactly two simple modules with injective dimension $1$, 
namely $S(2)$ and $S(3)$.
Their projective covers are
\[P(2)=[2,4) \text{ and }  P(3)=[3,5).\]
Note that $P(3)$ is injective and $\idim_AP(2)=1$.
By Lemma \ref{lem:2-gor}, the algebra $A$ is $2$-Gorenstein.

One can show that  $\gldim A=2$  and 
$\varepsilon(A)$ is semisimple.
Hence, $A$ is Auslander regular of global dimension $2$.
\end{exam}

\section{Odd Ext modules of simple modules}\label{sec:odd-ext}
In this section, we study the structure of odd Ext modules of simple modules 
over Auslander-Gorenstein Nakayama algebras. 
We obtain a module-theoretic strengthening of the dimension bound conjecture
of Klász, Kleinau and Marczinzik \cite{KKM26}.

\begin{thm} \label{thm:kkm}
Let $A$ be an Auslander-Gorenstein Nakayama algebra and let $S$ be a simple
$A$-module. For any odd integer $n\geq 1$, the following statements hold.
\begin{enumerate}
  \item $\Ext_A^n(S,A)\neq 0$ if and only if $\pdim_AS=n$;
  \item $\Ext_A^n(S,A)$ is either zero or a simple $A^{\op}$-module.
\end{enumerate}
\end{thm}

When $A$ is a split basic finite dimensional algebra over a field $k$, every
simple  $A^{\op}$-module is one-dimensional over $k$. 
Thus, Theorem~\ref{thm:kkm} implies the dimension bound conjectured in
\cite[Section 5]{KKM26}.

Let $A$ be a Nakayama algebra.
The base module $\nabla$ is an indecomposable module 
of minimal composition length such that 
the injective dimensions of $\top\nabla$ and 
$\tau\soc\nabla$ are different from $1$.

We need the following lemma.

\begin{lem} \label{lem:base-simple}
If $A$ is a $2$-Gorenstein Nakayama algebra and 
$\nabla$ is a base module with $\pdim_A\nabla=1$, 
then  $\nabla$ is a simple $A$-module.
\end{lem}

\begin{proof}
Assume that $\nabla$ is not simple.
Then $T=\tau\top\nabla$ is a composition factor of $\nabla$. 
The defining property  of $\nabla$ gives $\idim_AT=1$.
By Lemma \ref{lem:2-gor} $\idim_AP(T)$ is at most $1$.
Since $\pdim_A\nabla=1$,
Lemma \ref{lem:madsen} implies that  $\pdim_A\tau^{-1}T=1$.
By Lemma \ref{lem:bijection}, $P(T)$ is not injective. 
It follows that $\idim_AP(T)= 1$.

We note that  $\rad \nabla$ is a quotient module of $P(T)$.
If these two modules are equal, 
by Lemma \ref{lem:madsen}  $\pdim_AP(T)=1$.
This is impossible. 
It follows that $\rad \nabla$ is a proper quotient of $P(T)$. 
Then $S=\tau\soc\nabla$ is a composition factor of $P(T)$.
By Lemma \ref{lem:madsen} $\idim_AS=1$.
This contradicts  the defining property  of $\nabla$.
Therefore, $\nabla$ is a simple $A$-module.
\end{proof}

\begin{proof}[Proof of Theorem \ref{thm:kkm}]     
(1)
We argue by induction on $n$.
For $n=1$, this follows immediately from \eqref{eq:I1}.

Let $n\geq3$ be odd and assume the assertion in degree $n-2$.
If $\pdim_A S=0,1$, then
$\Ext_A^n(S,A)=0$ and $\pdim_A S\neq n$. 
We may assume that $\pdim_A S\geq 2$.

Let $\scE$ be the syzygy filtration category of $A$. By
Theorem~\ref{thm:nak-ausgor}, the algebra
$\varepsilon(A)$ is again
Auslander-Gorenstein, 
and $\scE$ is equivalent to $\mod\varepsilon(A)$.
Under this equivalence, the base module $\nabla(S)$ is simple.
Corollary~\ref{cor:ext-2} gives
\[
\Ext_A^n(S,A)
\cong
\Ext_{\scE}^{n-2}\bigl(\nabla(S),A_\varepsilon\bigr).
\]
The defining sequence of $\nabla(S)$
shows that $\nabla(S)$ is the second syzygy of $S$. 
Hence,
\[\pdim_{\scE}\nabla(S)=\pdim_A S-2.\]
Applying the induction hypothesis to $\varepsilon(A)$, we obtain
\begin{align*}
\Ext_A^n(S,A)\neq0
&\iff
\Ext_{\scE}^{n-2}\bigl(\nabla(S),A_\varepsilon\bigr)\neq0\\
&\iff
\pdim_{\scE}\nabla(S)=n-2\\
&\iff
\pdim_A S=n.
\end{align*}

(2)
Now assume that  $\pdim_A S=n$.
Let
\[
0\lto P_n\lto  P_{n-1}\lto \cdots\lto  P_0\lto  S\lto 0
\]
be the minimal projective resolution of $S$. Applying
$\Hom_A(-,A)$, we get
\[
\Ext_A^n(S,A)=\Coker(P_{n-1}^{*}\longrightarrow P_n^{*}).
\]
By the definition of the higher transpose,
there is a natural isomorphism 
\[
D\Ext_A^n(S,A)\cong D\Tr_n(S)\cong D\Tr \Omega^{n-1}(S).
\]

Write $n=2r+1$ and $T=\Omega^{2r}(S)$.
It suffices to show that $T$ is a simple $A$-module.
Then $D\Tr T$ is simple. 
Hence, $D\Ext_A^n(S,A)$ and thus $\Ext_A^n(S,A)$ is simple.

Set $\scE^0=\mod A$ and $\scE^i=\scE(\scE^{i-1})$ for $1\leq i\leq r$.
Then the iterated filtration category $\scE^{i}$ is equivalent to 
the category of all finitely generated $\varepsilon^i(A)$-modules.
Since $\pdim_AS=2r+1$, it follows that $\pdim_AT=1$.
By the construction of iterated syzygy filtration,
$T$ is a simple object of $\scE^{r}$.
Lemma \ref{lem:ext-comparison} gives
$\pdim_{\scE^i}T=1$ for all $0\leq i\leq r$.
Since $A$ is Auslander-Gorenstein, 
the algebra $\varepsilon^i(A)$ is $2$-Gorenstein.
Iterating  Lemma \ref{lem:base-simple}, 
we infer that $T$ is a simple $A$-module.
\end{proof}

Let $A$ be an Artin algebra.
For $n\geq 0$, 
a simple $A$-module $S$ is called $n$-perfect if 
\[
    \grade_AS=n=\pdim_AS<\infty,
\]
and $S$ is called   {$n$-regular} if, 
in addition, $\Ext^n_A(S,A)$ is a simple $A^{\op}$-module.

Recall the following result; see \cite[Proposition 5.2]{KKM26}.

\begin{lem} \label{lem:kkm}
    Let $A$ be an Auslander-Gorenstein Nakayama algebra.
    Then all simple $A$-modules of odd grade are perfect at their grade. 
\end{lem}

Combining Theorem \ref{thm:kkm} with  Lemma \ref{lem:kkm}, 
we have the following.

\begin{cor}
Let $A$ be an Auslander-Gorenstein Nakayama algebra.
Then all simple $A$-modules of odd grade are regular at their grade. 
\end{cor}

We have the following.

\begin{cor}\label{cor:odd-injective-terms}
Let $A$ be an Auslander-Gorenstein Nakayama algebra and 
let $n\geq 1$ be an odd integer.
For  any simple module $S$, the following statements hold.
\begin{enumerate}
    \item $\pdim_AS=n$  if and only if 
    $I(S)$ is a direct summand of $I^n(A)$;
    \item the multiplicity of $I(S)$ in $I^n(A)$ is at most $1$.
\end{enumerate}
\end{cor}

\begin{proof}
Take a minimal injective coresolution $I^*(A)$ of $A$. 
For every simple $A$-module $S$, minimality implies an isomorphism
\[
\Ext_A^n(S,A)\cong\Hom_A(S,I^n(A)).
\] 
Thus, $I(S)$ is a direct summand of $I^n(A)$ if and only if
$\Ext_A^n(S,A)$ is nonzero. 
By Theorem \ref{thm:kkm} this condition is equivalent to $\pdim_AS=n$.

Let $\mu_n(S)$ be the multiplicity of $I(S)$ in $I^n(A)$ 
and let $\Delta_S$ be the opposite of $\End_A(S)$. 
The above isomorphism gives
\[\mu_n(S)=\dim_{\Delta_S}\Ext_A^n(S,A).\]
By Theorem~\ref{thm:kkm}, the $A^{\op}$-module
$E=\Ext_A^n(S,A)$ is either zero or simple. 
Suppose that $E$ is  simple.
By Theorem \ref{thm:kkm}, 
$DE$ and $S$ are in the same connected component.
Since $A$ is a Nakayama algebra, 
the homothety homomorphism
\[f\colon \Delta_S\to \End_{A^{\op}}(E)\]
is an isomorphism.
Since $A$ is basic, $E$ has dimension zero or one over
$\Delta_S$. Consequently, $\mu_n(S)=1$ when
$\pdim_A S=n$, and $\mu_n(S)=0$ otherwise. 
Hence, the multiplicity of $I(S)$ in $I^n(A)$ is at most $1$.
\end{proof}

\begin{rmk}
The preceding corollary is analogous to the description of the minimal
injective coresolution of a commutative Gorenstein ring \cite{Bass63}, in
which the occurrence degree of an indecomposable injective module is
determined by the height of the corresponding prime ideal. In the present
setting, occurrence degrees instead reflect homological properties of simple
modules. The corollary characterizes the indecomposable injective modules
occurring in odd degrees, whereas the structure of the even-degree terms
remains unclear.
\end{rmk}

\end{document}